\documentclass{article}
\usepackage{amssymb,latexsym,amsfonts,amsmath,amsthm}
\usepackage{graphics,graphicx,epstopdf}
\usepackage{enumitem}
\usepackage[colorlinks=true,linkcolor=black,citecolor=blue,linktocpage=true]{hyperref}

\newtheorem{theorem}{Theorem}[section]
\newtheorem{lemma}[theorem]{Lemma}

\newtheorem*{remark}{Remark}

\newcommand{\N}{\mathbb{N}}

\newcommand{\1}{\mathbf{1}}
\newcommand{\dual}[1]{ {#1}^{\ast} }

\begin{document}
\title{Projectional Skeletons of Fourier Algebras}
\author{Onur Oktay}
\maketitle

\begin{abstract}
The preduals of $W^*$-algebras are 1-Plichko spaces \cite{BohataHamhalterKalenda16}. A natural question arises: does every predual possess a projectional skeleton (PS) $\{P_s:s\in J\}$ such that each $P_s^*$ is a conditional expectation? In this note, we answer this question affirmatively for the preduals of the group von Neumann algebras of locally compact groups.
\end{abstract}
%


\section{Introduction}\label{sec:intro}

A \emph{projectional skeleton (PS)} on a Banach space $X$ is a family of uniformly bounded linear projections $\{P_s: s\in J\}$ indexed by a directed set $J$ satisfying
\begin{enumerate}
\item $P_s(X)$ is a separable subspace of $X$ for all $s\in J$.
\item If $s\leq t$, then $P_s = P_sP_t = P_tP_s$.
\item If $(s_n)$ is an increasing sequence in $J$, then $\sup s_n = s\in J$ exists and \\
$\displaystyle P_sx = \lim_{n\to\infty} P_{s_n}x$ for each $x\in X$.
\item $\displaystyle \lim_{s\in J} P_sx = x$ for each $x\in X$.
\end{enumerate}
$\{P_s: s\in J\}$ is a $1$-PS if $\|P_s\|=1$ for all $s\in J$. 
$\{P_s: s\in J\}$ is \emph{commutative} if $$P_sP_t = P_{t\wedge s}=P_tP_s\hspace{4mm}\forall s,t\in J.$$
A Banach space $X$ is a ($1$-)Plichko space if it possesses a commutative ($1$-)PS.
$X$ is ($1$-)Plichko if and only if $X$ has a countably ($1$-)norming Markushevich basis.
For a Banach space,
\begin{center}
Separable $\Rightarrow$ WCG $\Rightarrow$ WLD $\Rightarrow$ Plichko.
\end{center}
Here, WCG stands for weakly compactly generated, and WLD stands for weakly Lindel\"{o}f determined. Please see \cite[Chapter 6]{DevilleGodefroyZizler93}, 
\cite{Kalenda99-2,Kalenda99-3,Kalenda00,Kubis09,HMVZ08,Kalenda20} and \cite[Chapter 17-19]{KakolKubisPellicer11} for a comprehensive treatment of Plichko spaces.

Let $M$ be a $W^*$-algebra with predual $M_*$. Let $P:M_*\to M_*$ be a bounded linear projection, $X=P(M_*)$, and $N=P^*(M)$. Clearly, $\dual{X}$ is linearly weak$^*$ isomorphic to $N$ via the dual of $P$ considered as $P:M_*\to X$.

If $P$ is contractive and $N$ is a $W^*$-subalgebra, then $P^*$ is a normal conditional expectation \cite{Tomiyama57}, in which case $P$ is completely contractive.
Generally, if $P$ is contractive, then $N$ is a JB$W^*$-triple with the triple product $$\{x,y,z\} = \frac{1}{2}P^*(xy^*z + zy^*x)$$ for $x,y,z\in N$ \cite[Theorem 2]{FriedmanRusso85}.
If $P$ is completely contractive, then $N$ is a $W^*$-TRO with the triple product $\{x,y,z\} = P^*(xy^*z)$ \cite{Youngson83}. 

Conversely, suppose $X$ is a closed subspace of $M_*$. If $X$ is isometrically isomorphic to a $W^*$-algebra predual, then there exists a contractive projection $P:M_*\to M_*$ onto $X$ \cite{Kirchberg93}. If $X$ is completely isometrically isomorphic to a $W^*$-TRO predual, then there exists a completely contractive projection $P:M_*\to M_*$ onto $X$ \cite{NgOzawa02}. 

Hence, the projectional structure of $M_*$ and the algebraic structure of $M$ are strongly related. Regarding the projectional structure, it was shown that the preduals of $W^*$-algebras (generally JB$W^*$-triples) are 1-Plichko spaces \cite{BohataHamhalterKalenda16,BohataHamhalterKalendaPeraltaPfitzner17}. Although $1$-PS is in harmony with the JB$W^*$-triple structure of $M$, one could expect more from the PS of a $W^*$-algebra predual. Hence, a natural question arises: does $M_*$ have a 1-PS $\{P_s:s\in J\}$ such that each $P_s^*$ is a conditional expectation?

In this article, we provide a positive answer to this question for the preduals of the group von Neumann algebras of locally compact groups. In Section~\ref{sec:prelim} we collect preliminaries that we will need for the proof of our main result Theorem~\ref{thm:main}. In Section~\ref{sec:main}, we present and prove Theorem~\ref{thm:main}.


\section{Preliminaries}\label{sec:prelim}

We use the definitions and notation in \cite{Eymard64,KaniuthLau18} related to the Fourier algebra $A(G)$, Fourier-Stieltjes algebra $B(G)$, and the von Neumann algebra $VN(G)$ of a locally compact group (LCG) $G$. We refer the reader to \cite{KaniuthLau18} for the properties of these algebras that we use but do not provide a proof in this text.

\subsection{Structure of LCG}

The structure of LCG is described by the strong form of the Gleason-Yamabe theorem, e.g. \cite[Theorem 6.0.11]{gsm153}.
\begin{theorem}[Gleason-Yamabe]\label{thm:Gleason-Yamabe}\normalfont
Let $G$ be a LCG. There exists an open subgroup $O$ with the property that every neighborhood of the identity contains a compact normal subgroup $K$ of $O$ such that $O/K$ is a Lie group.
\end{theorem}

Let $G_0$ denote the connected component of $G$, and $q:G\to G/G_0$ be the quotient homomorphism. $G/G_0$ is a totally disconnected LCG, so contains a compact open subgroup $\tilde{O}$ by van Dantzig's theorem. Looking closer at the proof of Theorem~\ref{thm:Gleason-Yamabe}, we can choose $O=q^{-1}(\tilde{O})$, so $O/G_0$ is compact.

Given $K$ as in Theorem~\ref{thm:Gleason-Yamabe}, let $(O/K)_0$ denote the connected component of the identity of $O/K$. Since $O/G_0$ is compact, then so is $(O/K)/(O/K)_0$.
On the other hand, $(O/K)_0$ is a connected Lie group, so it is second countable and open in $O/K$. Being compact and discrete, $(O/K)/(O/K)_0$ is finite. Hence, since $(O/K)_0$ is second countable, $O/K$ is also second countable.

Suppose $K_1$, $K_2$ are two compact normal subgroups such that $O/K_1$ and $O/K_2$ are second countable. Clearly $K_1\cap K_2$ and $K_1K_2$ are also a compact normal subgroups. $O/(K_1\cap K_2)$ is second countable since it is embedded in $O/K_1 \times O/K_2$. $O/(K_1K_2)$ is second countable since the quotient map $q:O/K_1\to O/(K_1K_2)$ is open, continuous, and surjective. Consequently, 
\begin{lemma}\normalfont\label{lem:K-lattice}
Let $O$ be as in Theorem~\ref{thm:Gleason-Yamabe}. 
The set $\mathcal{K}$ of all compact normal subgroups $K$ of $O$, for which $O/K$ is second countable, is a lattice with 
$$K_1\leq K_2\Leftrightarrow K_1\supseteq K_2, \hspace{6mm}
K_1\wedge K_2 = K_1K_2, \hspace{6mm}
K_1\vee K_2=K_1\cap K_2.$$
Furthermore, $\mathcal{K}$ is $\sigma$-join complete, i.e., every countable increasing chain in $\mathcal{K}$ has a supremum in $\mathcal{K}$.
\end{lemma}
\begin{proof}
$\mathcal{K}$ is a lattice by the preceding paragraph. 
For $\sigma$-join completeness, let $(K_n)_{n\in\N}$ be a sequence of compact normal subgroups such that $K_n\supseteq K_{n+1}$ for all $n\in\N$. Let $\displaystyle K = \bigcap K_n$. We will show that $O/K$ is homeomorphically isomorphic to the inverse limit 
$$P = \left\{(x_nK_n)_{n\in\N}\in\prod_{n\in\N} O/K_n : \phi_n(x_{n+1}K_{n+1}) = x_nK_n \hspace{4mm}\forall n\in\N\right\}$$
where $\phi_n: O/K_{n+1}\to O/K_n$ are the quotient homomorphisms. \\ Note that since each $O/K_n$ is second countable, so is $\prod_{n\in\N} O/K_n$, and so is $P$.

Clearly $K$ is a compact normal subgroup.
Let $q:O\to O/K$ and \\ $q_n:O\to O/K_n$ denote the quotient maps.

Define the map $f:O/K\to\prod O/K_n$ by $f(xK) = (xK_n)_{n\in\N} = (q_n(x))_{n\in\N}$.
$$xK = yK\Rightarrow \forall n\in\N\hspace{4mm}y^{-1}x\in K_n\Rightarrow f(xK) = f(yK)$$
so $f$ is well-defined.
$$f(xK) = f(yK)\Rightarrow\forall n\in\N\hspace{4mm} xK_n = yK_n\Rightarrow y^{-1}x\in K\Rightarrow xK = yK$$ 
so $f$ is injective.

In order to show $f(O/K)=P$, let $(x_nK_n)\in P$. 
$$x_{n+1}K_{n+1}\supseteq x_{n+1}K_n = \phi_{n+1}(x_{n+1}K_{n+1}) = x_nK_n\hspace{4mm}\forall n\in\N$$ 
Since each $K_n$ is compact, $\bigcap x_nK_n \neq\varnothing$. Then, $f(xK) = (x_n K_n)$ for $x\in\bigcap x_nK_n$. Thus, $P\subseteq f(O/K)$. The reverse inclusion $P\subseteq f(O/K)$ is clear. Hence, $f(O/K)=P$.

$f$ is continuous since each $q_n$ is continuous.
Next, let $E\subseteq O/K$ be an open neighborhood of identity. Then $q^{-1}(E)\subset O$ is open, and $\bigcap K_n = K\subset q^{-1}(E)$. Thus,
$$\exists N\in\N\hspace{4mm}K_N\subset q^{-1}(E).$$ 
%
$V=\{(x_nK_n)\in P: x_N\in K_N\}$ is an open neighborhood of identity in $P$, and $V\subseteq f(E)$.
Consequently, $f$ is an open map.
\end{proof}

\begin{lemma}\normalfont\label{lem:J-lattice}
Let $O$ be as in Theorem~\ref{thm:Gleason-Yamabe}. The set $\mathcal{J}$ of all subgroups $B$ of $G$, for which $B\supseteq O$ and $B/O$ is countable, is a $\sigma$-join complete lattice with 
$$B_1\leq B_2\Leftrightarrow B_1\subseteq B_2, \hspace{6mm}
B_1\wedge B_2 = B_1\cap B_2,$$
$B_1\vee B_2$ is the smallest subgroup that contains both $B_1$ and $B_2$.
\end{lemma}
\begin{proof}
Let $B_1*B_2$ denote the free product of $B_1$ and $B_2$. 
First, there is a surjective map $(B_1*B_2)/O\to (B_1\vee B_2)/O$, 
and an injective map $(B_1*B_2)/O\to \bigcup_{n\in\N}(B_1/O)^n \times\bigcup_{n\in\N} (B_2/O)^n$ which is countable. Thus, $(B_1\vee B_2)/O$ is countable.

Second, let $(B_n)_{n\in\N}$ be an increasing sequence of subgroups. Then, 
$\bigvee_{n\in\N} B_n = \bigcup_{n\in\N} B_n$ 
and
$(\bigcup_{n\in\N} B_n)/O = \bigcup_{n\in\N} B_n/O$ is countable
\end{proof}

\begin{lemma}\label{lem:Haar-limit}\normalfont
Let $G$ be a LCG, and $(K_i)_{i\in I}$ be a net of compact groups such that $i\leq j\Rightarrow K_i\supseteq K_j$.
Let $\mu_i$ be the normalized Haar measure on $K_i$. 
Let $K=\bigcap_{i\in I} K_i$ and $\mu$ be the normalized Haar measure on $K$. 
Then $\mu_i\to\mu$ in weak$^*$ topology of $M(G)$.
\end{lemma}
\begin{proof}
Clearly $K$ is a non-empty compact subgroup of $G$. 
By the Banach-Alaoglu theorem, the unit ball of $M(G)$ is weak$^*$ compact, so the net $(\mu_i)$ has a cluster point $\nu$. Clearly $\nu(K)=1$. WLOG, let $(\mu_i)$ denote the subnet converging to $\nu$.
$$\textrm{supp}(\nu) \subseteq \bigcap_{i\in I} \textrm{supp}(\mu_i) = K$$
since $\textrm{supp}(\mu_i) = K_i$. For any $x\in K$ and $f\in C_c(G)$
$$\int_G f(xy)\ d\nu(y)
= \lim_{i\in I} \int_G f(xy)\ d\mu_i(y) 
= \lim_{i\in I} \int_G f(y)\ d\mu_i(y) 
= \int_G f(y)\ d\nu(y)$$
so $\nu$ is translation invariant.

Consequently, $\nu = \mu$ by the uniqueness of the normalized Haar measure.
\end{proof}

\subsection{Projections on $A(G)$}

$A(G)$ is separable if and only if $G$ is second countable. Putting this into center, we will gather the interplay between $A(G)$ and $A(H)$ when $H$ is an open subgroup or a compact normal subgroup in the following lemmas. 

\begin{lemma}\label{lem:1}\normalfont
Let $G$ be a LCG, $K$ a compact subgroup with Haar measure $\mu$. Let $\pi:G\to B(H)$ be a unitary representation. Then, $Q:H\to H$ defined by $$Q = \int_K \pi(y) d\mu(y)$$ is an orthogonal projection onto $H_K = \{\xi\in H: \pi(a)\xi=\xi\hspace{2mm}\forall a\in K\}$, which satisfies $\pi(x)Q = Q\pi(x)$ for all $x\in N_G(K) = \{x\in G: xKx^{-1}\subseteq K\}$. 

Moreover, if $Q_1, Q_2$ are two projections related to two compact subgroups $K_1$ and $K_2$ such that $K_2\subseteq N_G(K_1)$, then $Q_1Q_2 = Q_2Q_1$ is an orthogonal projection onto $H_{K_1K_2}$.
\end{lemma}
\begin{proof}
The operator $Q$ is defined in the weak sense: for any $\xi,\eta\in H$
$$\langle Q\xi,\eta\rangle = \int_K \langle\pi(y)\xi,\eta\rangle\ d\mu(y)$$ and since the integrand is a continuous function on a compact set, the integral is well-defined and bounded by $\|\xi\|\|\eta\|$.
%
Next, 
$$Q^2 
= \int_K \int_K \pi(xy)\ d\mu(y)d\mu(x)
= \int_K \int_K \pi(y)\ d\mu(y)d\mu(x)
=Q,$$
and
$$Q^* = \int_K \pi(y)^*\ d\mu(y) = \int_K \pi(y^{-1})\ d\mu(y) = Q.$$
If $\xi\in H_K$, then $Q\xi = \int_K \pi(y)\xi\ d\mu(y) = \xi$, so $\xi\in Q(H)$. 
If $\xi\in Q(H)$, then $\pi(a)\xi = \pi(a)Q\xi = Q\xi = \xi$ for all $a\in K$, so $\xi\in H_K$. Thus $Q(H)=H_K$.

If $x\in N_G(K)$, then by the uniqueness of the Haar measure
$$\pi(x)Q\pi(x)^{-1} 
= \int_{K} \pi(xyx^{-1})\ d\mu(y)
= \int_K \pi(y)\ d\mu(y)
= Q$$
Lastly, if $Q_1,Q_2$ are two projections related to two compact subgroups $K_1,K_2$ such that $K_1\subseteq N_G(K_2)$, then $$Q_1Q_2
= \int_{K_2} Q_1\pi(y)\ d\mu_2(y) 
= \int_{K_2} \pi(y)Q_1\ d\mu_2(y)
= Q_2Q_1.$$
Consequently, $Q_1Q_2$ is an orthogonal projection onto $H_{K_1}\cap H_{K_2} = H_{K_1K_2}$.
\end{proof}

\begin{lemma}\label{lem:2}\normalfont
Let $G$ be a LCG, $K$ be a compact subgroup.
$$A(G)^K = \{u\in A(G) : u(a^{-1}x) = u(x)\hspace{4mm}\forall a\in K\hspace{2mm} \forall x\in G\}$$ 
is a $1$-complemented subspace of $A(G)$. The map $P:A(G)\to A(G)^K$ 
\begin{equation}\label{def:P_K}
Pu(x) = \int_K u(a^{-1}x)\ d\mu(a)
\end{equation}
is a contractive projection, where $\mu$ is the Haar measure on $K$.

If $K$ is also normal, then $A(G)^K$ is isometrically isomorphic to the Fourier algebra of the quotient group $A(G/K)$.

Moreover, if $P_1, P_2$ are two projections related to two compact subgroups $K_1$,$K_2$ satisfying $K_2\subseteq N_G(K_1)$, then $P_1P_2 = P_2P_1$ is a contractive projection onto $A(G)^{K_1K_2}$.
\end{lemma}
\begin{proof}
Let $\lambda$ denote the regular representation of $G$. Let $Q:L^2(G)\to L^2(G)$ be defined by 
\begin{equation}
Q = \int_K \lambda(y)\ d\mu(y) \label{def:Q}
\end{equation} 
as in Lemma~\ref{lem:1}, where $\mu$ is the unique normalized Haar measure on $K$. $Q$ is an orthogonal projection satisfying $\lambda(a)Q = Q$ for all $a\in K$ by Lemma~\ref{lem:1}.

Indeed $P(A(G))\subseteq A(G)$. 
For every $u \in A(G)$ there exist $\xi,\eta \in L^2(G)$ such that
$u(x) = \langle \lambda(x)\xi, \eta \rangle$, so
$$Pu(x) 
= \int_K \langle \lambda(a^{-1}x)\xi, \eta \rangle \ d\mu(a)
= \int_K \langle \lambda(x)\xi, \lambda(a)\eta \rangle \ d\mu(a)
= \langle \lambda(x)\xi, Q\eta\rangle
$$  
Further, $Pu\in A(G)^K$ since
$$Pu(a^{-1}x)
= \langle \lambda(a^{-1}x)\xi, Q\eta\rangle
= \langle \lambda(x)\xi, \lambda(a)Q\eta\rangle
= \langle \lambda(x)\xi, Q\eta\rangle
= Pu(x)$$
for all $a\in K$, $x\in G$. Also,
$$ P^2u(x)
= \int_K Pu(a^{-1}x) \ d\mu(a)
= \int_K Pu(x) \ d\mu(a) 
= Pu(x)
$$
Thus, $P$ is a projection onto $A(G)^K$. $\|P\| = 1$ since
$$\|Pu\|_{A(G)} 
\leq \inf\{\|\xi\|\|Q\eta\|: u = \langle\lambda(\cdot)\xi,\eta\rangle\}
\leq \inf\{\|\xi\|\|\eta\|: u = \langle\lambda(\cdot)\xi,\eta\rangle \}
= \|u\|_{A(G)}$$ 

If $K$ is also normal, then let $q:G\to G/K$ be the quotient homomorphism. Clearly, the map $u\to u\circ q$ is an isometric algebra isomorphism from $A(G/K)$ onto $A(G)^K$, see also \cite[Proposition 2.4.2]{KaniuthLau18}.

Finally, if $P_i$ are related to compact normal subgroups $K_i$ satisfying $K_2\subseteq N_G(K_1)$, then for $u(x) = \langle \lambda(x)\xi,\eta \rangle$
$$P_1P_2u(x) 
= \langle \lambda(x)\xi, Q_1Q_2\eta\rangle
= \langle \lambda(x)\xi, Q_2Q_1\eta\rangle
= P_2P_1u(x)
$$ by Lemma~\ref{lem:1}, where $Q_i$ are defined as in \eqref{def:Q}.
\end{proof}

\begin{lemma}\label{lem:3}\normalfont
Let $G$ be a LCG, $H$ be an open subgroup. For $c\in(H\backslash G)$ a right coset, let $P_{c}:A(G)\to A(G)$ be defined by $P_{c}u = \1_{c}u$, where $\1_{c}$ is the characteristic function of $c$. Then, 
\begin{enumerate}[label=\roman*.]
\item \label{itm:lem3_1} $P_H(A(G))$ and $A(H)$ are isometrically isomorphic Banach algebras. 
\item \label{itm:lem3_2} $P_{c}$ is a contractive projection, $P_{H}$ is an algebra homomorphism. 
\item \label{itm:lem3_3} $P_{c_1}P_{c_2} = 0$ if $c_1\neq c_2$. Consequently, for each $u\in A(G)$ and $S\subseteq(H\backslash G)$ 
$$ 
\sum_{c\in S}P_cu = u \1_{\cup S}
\hspace{4mm}\textrm{ and }\hspace{4mm}
\|\sum_{c\in S}P_cu\| = \sum_{c\in S}\|P_cu\|
$$ 
where the sums converge as nets of finite partial sums.

\item \label{itm:lem3_4} $P^*_c:VN(G)\to VN(G)$ is given explicitly by 
$$P_c^*(\lambda(x)) = \left\{
\begin{array}{ll} 
\lambda(x) &\textrm{if } x\in c\\ 
0 &\textrm{otherwise}\\ 
\end{array}\right.$$
where $\lambda$ is the left regular representation of $G$. The range of $P_c^*$ is the weak$^*$ closure of the span of $\{\lambda(a): a\in c\}$, i.e.,
$$P_c^*(VN(G)) = \{T\in VN(G): \textrm{supp}(T)\subseteq c \}$$
 
\item \label{itm:lem3_5} $P^*_H$ is a normal conditional expectation onto the subalgebra generated by $\{\lambda(a): a\in H\}$, which is isomorphic to $VN(H)$. If $c=Hx$, then
$$P_{c}^*(T) = P_H^*\big(T\lambda(x^{-1})\big) \lambda(x)$$ 
\end{enumerate}
\end{lemma}
\begin{remark}
$1.$ Pick $x_c\in c$ from each right coset $c\in H\backslash G$. Then, it is clear from above that $\{\lambda(x_c):c\in H\backslash G\}$ is a Pimsner-Popa basis for the pair $VN(H)\subset VN(G)$.

$2.$ The range of $P^*_c$ is a $W^*$-TRO with the triple product $\{x,y,z\}=xy^*z$, so the projections $P_c$ are not only contractive but completely contractive \cite{NgOzawa02}.
\end{remark}
\begin{proof}
$\1_{c}\in B(G)$ by Host's idempotent theorem \cite{Host86}, and $A(G)$ is an ideal in $B(G)$. Thus $P_c(A(G))\subseteq A(G)$. 
For \eqref{itm:lem3_1} see \cite[Proposition 2.4.1]{KaniuthLau18}. 
\eqref{itm:lem3_2} is clear since $\|\1_c\|=1$ by \cite[Theorem 2.1]{IlieSpronk05}. 
The rest follow by straightforward duality arguments.
\end{proof}

\begin{lemma}\label{lem:4}\normalfont
Let $G$ be a LCG, $H$ be an open subgroup, $K$ be a normal compact subgroup of $H$ with Haar measure $\mu$. Let $P_K:A(G)\to A(G)$ be as in \eqref{def:P_K}.
Let $s=\bigcup_{i\in I}Hx_i$ be a union of right cosets and $\1_s$ be its characteristic function. Then, $ P_K(\1_s u) = \1_s P_Ku$ for all $u\in A(G)$. 
\end{lemma}
\begin{proof}
Clearly $y\in Hx \Leftrightarrow a^{-1}y\in Hx$ for all $a\in K$, so 
$\1_s(a^{-1}y) = \1_s(y)$ for all $a\in K$. Thus,
$$P_K(\1_s u)(x)
= \int_K \1_s(a^{-1}x) u(a^{-1}x)\ d\mu(a)
= \1_s(x) \int_K u(a^{-1}x)\ d\mu(a)
= \1_s(x) P_Ku(x)$$
\end{proof}


\section{Main theorem}\label{sec:main}

\begin{lemma}\label{lem:I-lattice}\normalfont
Let $G$ be a LCG, and open subgroup $O\subseteq G$ be given by Theorem~\ref{thm:Gleason-Yamabe}. Let $\mathcal{K}$ be as in Lemma~\ref{lem:K-lattice} and $\mathcal{J}$ be as in Lemma~\ref{lem:J-lattice}.
Let us define a partial order on the set 
$\mathcal{I} = \{(B,K): B\in\mathcal{J},\hspace{2mm}K\in\mathcal{K}\}$ 
by 
$$(B_1,K_1)\leq (B_2,K_2) \Leftrightarrow B_1\subseteq B_2\textrm{ and } K_1\supseteq K_2.$$ 
$\mathcal{I}$ is a $\sigma$-join complete lattice with $(B_1,K_1)\vee (B_2,K_2) = (B_1\vee_\mathcal{J} B_2, K_1\cap K_2)$ and $(B_1,K_1)\wedge (B_2,K_2) = (B_1\cap B_2, K_1K_2)$. 
\end{lemma}
\begin{proof}
$\mathcal{K}$ and $\mathcal{J}$ are $\sigma$-join complete lattices by Lemmas~\ref{lem:K-lattice} and \ref{lem:J-lattice}. The cartesian product of two $\sigma$-join complete lattices is again a $\sigma$-join complete lattice \cite{Gratzer78}.
\end{proof}

\begin{theorem}\label{thm:main}\normalfont
Let $G$ be LCG. 
Let $\mathcal{I}$ be as in Lemma~\ref{lem:I-lattice}.
For each $(B,K)\in\mathcal{I}$, define $P_{(B,K)}:A(G)\to A(G)$ by
\begin{equation}\label{def:P_BK}
P_{(B,K)}u(x) = \1_B(x) \int_K u(a^{-1}x)\ d\mu(a)
\end{equation}
where $\mu$ is the Haar measure on $K$.
Then, $\{P_i : i\in\mathcal{I}\}$ is a commutative projectional skeleton such that $P^*_i$ is a conditional expectation for each $i\in\mathcal{I}$.
\end{theorem}
\begin{proof}
Let $Q_K:L^2(G)\to L^2(G)$ be defined as in \eqref{def:Q}. Let $P_K:A(G)\to A(G)$
$$P_Ku(x) = \langle \lambda(x)\xi,Q_K\eta\rangle$$
where $u(x) = \langle \lambda(x)\xi,\eta\rangle$ and $\xi,\eta\in L^2(G)$. Then,
$$P_{(B,K)}u(x) 
= \1_B(x) P_Ku(x)
$$
$P_K$ is a contractive projection onto $A(G)^K$ by Lemma~\ref{lem:2}. 
On the other hand, $B$ is an open subgroup of $G$, so $\1_B\in B(G)$ by Host's idempotent theorem \cite{Host86} and $\|\1_B\|_{B(G)}=1$ by \cite[Theorem 2.1]{IlieSpronk05}. Since $A(G)$ is an ideal in $B(G)$, $$P_{(B,K)}(A(G)) = \1_B A(G)^K\subseteq A(G).$$  
$\|P_{(B,K)}u\|\leq\|\1_B\|\|P_Ku\|\leq\|u\|$ so $\|P_{(B,K)}\|=1$.

Next, by Lemma~\ref{lem:3}.\ref{itm:lem3_3},
$$P_{(B,K)}(A(G)) = \big(\bigoplus_{c\in O\backslash B} \1_c A(G)^K \big)_{\ell^1}$$
is separable since the set of right cosets $O\backslash B$ is countable and each summand above is separable. Indeed,
let $R_a:A(G)\to A(G)$, $R_au(x) = u(xa^{-1})$ denote the right translation. For $c=Oa$, the restriction of $R_a$ is an isometric linear isomorphism from $\1_O A(G)^K$ onto $\1_c A(G)^K$, $\1_O A(G)^K$ is isometrically isomorphic to $A(O)^K$ by Lemma~\ref{lem:3}.\ref{itm:lem3_1}, and $A(O)^K$ is isometrically isomorphic to $A(O/K)$ by Lemma~\ref{lem:2}. $A(O/K)$ is separable since $O/K$ is second countable.

Second, for $(B_1,K_1), (B_2,K_2)\in\mathcal{I}$
\begin{eqnarray*}
P_{(B_1,K_1)}P_{(B_2,K_2)}u
&=& \1_{B_1}P_{K_1}(\1_{B_2}P_{K_2}u) \\
&=& \1_{B_1}\1_{B_2} P_{K_1}P_{K_2}u \\
&=& \1_{B_1\cap B_2} P_{K_1K_2}u
\end{eqnarray*}
by Lemma~\ref{lem:4} and Lemma~\ref{lem:2}. Thus, $P_{(B_1,K_1)}$ and $P_{(B_2,K_2)}$ commute. In particular, 
$$(B_1,K_1)\leq (B_2,K_2) \Rightarrow 
P_{(B_1,K_1)}P_{(B_2,K_2)} = P_{(B_1,K_1)} = P_{(B_2,K_2)}P_{(B_1,K_1)}.$$

Third, if $((B_n,K_n))_{n\in\N}$ is an increasing chain in $\mathcal{I}$, let 
$$B=\bigcup_{n\in\N}B_n \hspace{3mm}\textrm{and}\hspace{3mm} K=\bigcap_{n\in\N}K_n .$$
$B\in\mathcal{J}$ by Lemma~\ref{lem:J-lattice} and $K\in\mathcal{K}$ by Lemma~\ref{lem:K-lattice}.
If $\mu_n$, $\mu$ are the normalized Haar measures of $K_n$ and $K$, respectively, then $\mu_n\to\mu$ in the weak$^*$ topology of $M(K_0)$ by Lemma~\ref{lem:Haar-limit}. 
Thus, for each $u\in A(G)$
$$
\lim_{n\to\infty} P_{(B_n,K_n)}u
= \lim_{n\to\infty} \1_{B_n} P_{K_n}u
= \1_B P_{K}u
$$

Fourth, $\bigcap_{K\in\mathcal{K}}=\{0\}$, so the net $(\mu_K)_{K\in\mathcal{K}}$ is weak$^*$ convergent to Dirac $\delta_0$ by Lemma~\ref{lem:Haar-limit}. Thus, $\displaystyle\lim_{K\in\mathcal{K}}P_Ku = u$, and so
$$
\lim_{(B,K)\in\mathcal{I}} P_{(B,K)}u
= \lim_{(B,K)\in\mathcal{I}} \1_B P_Ku
= \1_{G} u
= u
$$
for each $u\in A(G)$.

Finally, the range of $P^*_{(B,K)}$ is
\begin{eqnarray*}
P^*_{(B,K)}(VN(G)) 
&=& \{T\in VN(G):\textrm{supp}(T)\subseteq B,\hspace{2mm}T=\lambda(\mu)T\}\\
&=& \{T\in VN(G):\textrm{supp}(T)\subseteq B,\hspace{2mm}T=\lambda(a)T\hspace{2mm} \forall a\in K\}
\end{eqnarray*} 
which is a von Neumann subalgebra of $VN(G)$. Thus, $P^*_{(B,K)}$ is a conditional expectation by Tomiyama's theorem \cite{Tomiyama57}.
\end{proof}

\bibliography{vn1.bbl}

\end{document}